\documentclass[12pt,oneside,reqno]{amsart}
\usepackage{amsfonts}
\usepackage{amssymb}
\usepackage{enumerate}
\usepackage{graphicx}
\usepackage[dvipsnames]{xcolor}

\makeatletter
\@namedef{subjclassname@2010}{  \textup{2010} Mathematics Subject Classification}
\makeatother
\newtheorem{theorem}{Theorem}[section]
\newtheorem{lemma}[theorem]{Lemma}
\newtheorem{proposition}[theorem]{Proposition}
\newtheorem{corollary}[theorem]{Corollary}

\theoremstyle{definition}

\numberwithin{equation}{section}
\DeclareMathOperator{\inter}{int}
\DeclareMathOperator{\diam}{diam}

\DeclareMathOperator{\co}{co}
\DeclareMathOperator{\ch}{ch}

\makeatletter
\@namedef{subjclassname@2020}{  \textup{2020} Mathematics Subject Classification}
\makeatother

\begin{document}
\title[Wolff-Denjoy theorems]{Wolff--Denjoy theorems in geodesic spaces}
\author[A. Huczek]{Aleksandra Huczek}
\address{Department of Mathematics, Pedagogical University of Krakow,
PL-30-084 Cracow, Poland}
\email{aleksandra.huczek2@doktorant.up.krakow.pl}
\author[A. Wi\'{s}nicki]{Andrzej Wi\'{s}nicki}
\address{Department of Mathematics, Pedagogical University of Krakow,
PL-30-084 Cracow, Poland}
\email{andrzej.wisnicki@up.krakow.pl}
\date{}

\begin{abstract}
We show a Wolff--Denjoy type theorem in complete geodesic spaces in the
spirit of Beardon's framework that unifies several results in this area. In
particular, it applies to strictly convex bounded domains in $\mathbb{R}^{n}$
or $\mathbb{C}^{n}$ with respect to a large class of metrics including
Hilbert's and Kobayashi's metrics. The results are generalized to
$1$-Lipschitz compact mappings in infinite-dimesional Banach spaces.
\end{abstract}

\maketitle

\section{Introduction}

The classical Wolff--Denjoy theorem asserts that if $f:\Delta \rightarrow
\Delta $ is a holomorphic map of the unit disc $\Delta \subset \mathbb{C}$
without a fixed point, then there is a point $\xi \in \partial \Delta $ such
that the iterates $f^{n}$ converge locally uniformly to $\xi $ on $\Delta $.
There is a wide literature concerning various generalizations of this
theorem. In complex variables, the first results are due to Herv\'{e} \cite
{He} and Abate \cite{Ab}. Wolff--Denjoy type theorems for fixed-point free $
1 $-Lipschitz maps on Hilbert's metric spaces were studied by Beardon \cite
{Be2} and were further developed by Karlsson \cite{Ka1, Ka2}, Karlsson and
Noskov \cite{KaNo}, Nussbaum \cite{Nu}, and recently by Lemmens et al. \cite
{LLNW}, to mention only a few papers.

Beardon argued in \cite{Be1} that the Wolff--Denjoy theorem is a purely
geometric result depending essentially only on the hyperbolic properties of
a metric and showed its counterpart for a large class of negatively curved
Riemannian spaces. In \cite{Be2}, he proposed a general approach that
avoided any smoothness assumptions on the metric space under consideration.
In particular, he proved a Wolff--Denjoy type theorem for contractive maps
on a strictly convex bounded domain $D\subset \mathbb{R}^{n}$ with Hilbert's
metric. The contractivity assumption was relaxed later by Karlsson
\cite{Ka1}. In the case of a convex bounded domain $D$, it was proved in \cite{KaNo}
that the attractor (in the norm topology) $\Omega _{f}=\bigcup_{x\in
	D}\omega _{f}(x)$ of a fixed point free nonexpansive (i.e., $1$-Lipschitz)
mapping $f:D\rightarrow D$ is a star-shaped subset of $\partial D$. This has
led to a conjecture formulated by Karlsson and Nussbaum that $\co \Omega
_{f}\subset \partial D$ (see \cite{Ka3, Nu}).

In 2012, Budzy\'{n}ska \cite{Bu1} (see also \cite{Bu2, BKR2, BKR3} for
infinite dimensional generalizations) obtained the Wolff--Denjoy theorem for
nonexpansive maps on a strictly convex bounded domain $D\subset \mathbb{C}
^{n}$ with the Kobayashi distance. Budzy\'{n}ska's arguments were sharpened
by Abate and Raissy in \cite{AbRa}. The question naturally arises, to what
extent can Beardon--Karlsson's and Budzy\'{n}ska--Abate--Raissy's arguments
interact with each other to gain a deeper insight into this problem.

The central observation in \cite{Be2} is that a Wolff--Denjoy type theorem
holds for any proper metric space $(Y,d)$ satisfying the following
condition:
\begin{equation}
d(x_{n},y_{n})- \max \{d(x_{n},w),d(y_{n},w)\} \rightarrow \infty 
\tag{B} 
\label{B}
\end{equation}
for any sequences $\{x_{n}\}$ and $\{y_{n}\}$ in $Y$ converging to distinct
points in $\partial Y$ and any $w\in Y$ (see Section 2 for more details).
The classical Wolff--Denjoy theorem follows then as a particular case of a
general theorem by using the cosine formula in hyperbolic geometry.
Condition (\ref{B}) is also satisfied for bounded strictly convex domains of
$\mathbb{R}^{n}$ with the Hilbert metric as a consequence of (a variant of)
the intersecting chord theorem (see \cite{Be2}). In Hilbert's geometry, the
geodesics are straight-line segments and we cannot expect such a regular
behaviour for strictly convex domains of $\mathbb{C}^{n}$ with the Kobayashi
distance. The arguments in the complex case in \cite{AbRa, Bu1} are more
intricate, and it is natural to try to refine their basic features in the
spirit of Beardon's work.

This is what we do in this paper. We extend Beardon's framework and show, in
particular, that in proper geodesic spaces condition (\ref{B}) can be
relaxed to the following:
\begin{equation}
d(x_{n},y_{n})-d(y_{n},w)\nrightarrow -\infty
\tag{B'}
\label{B'}
\end{equation}
for any sequences $\{x_{n}\}$ and $\{y_{n}\}$ in $Y$ converging to distinct
points in $\partial Y$ and any $w\in Y$ (see Theorem \ref{main2}). Section 3
contains a few variations on this theme, giving general results of
Wolff--Denjoy type in geodesic spaces.

Section 4 is devoted to the study of bounded strictly convex\ domains in
finite dimensional spaces. It follows from Proposition \ref{A3star} that
condition (\ref{B'}) is much less restrictive than Beardon's condition (\ref
{B}), and holds for any Euclidean metric for which balls are (linearly)
convex. Consequently, Theorem \ref{main_Rn} shows a general Wolff--Denjoy
type theorem for bounded strictly convex\ domains in a finite dimensional
space that, in particular, unifies Beardon's and Budzy\'nska's results
regarding Hilbert's and Kobayashi's metrics. Quite surprisingly, it holds
for any complete geodesic space with convex balls whose topology coincides
with the norm topology, and any `hyperbolic' property of a metric is not needed.
In Section 5 we extend the foregoing results to the case of $1$-Lipschitz compact
mappings in infinite dimensional spaces.

Section 6 discusses the recent characterization of geodesic boundedness of
spaces of negative curvature (see \cite{Pi1}). We apply Karlsson's
Wolff-Denjoy type theorem for Gromov hyperbolic spaces to give a more direct
proof of this interesting result.

\section{Preliminaries}

Let $(Y,d)$ be a metric space. If $x\in Y$ and $r>0$, then $B(x,r)$ and $
\bar{B}(x,r)$ denote the open and closed balls with midpoint $x$ and radius $
r.$ A curve $\sigma :[a,b]\rightarrow Y$ is a \textit{geodesic (segment)} if
$d(\sigma (t_{1}),\sigma (t_{2}))=\left\vert t_{1}-t_{2}\right\vert $ for
all $t_{1},t_{2}\in \lbrack a,b]$. By abuse of notation, the same name is
used for the image $\sigma ([a,b])\subset Y$ of $\sigma $, denoted by $
[\sigma (a),\sigma (b)].$ We say that $Y$ is a \textit{geodesic space} if
every two points of $Y$ can be joined by a geodesic. $Y$ is \textit{proper} if every
closed ball $\bar{B}(x,r)$ is compact. It follows from the Hopf-Rinow
theorem that every locally compact complete geodesic space is proper.

We are interested in the iteration of nonexpansive mappings acting on
geodesic spaces. A map $f:Y\rightarrow Y$ is called \textit{nonexpansive} if
for any $x,y\in Y,$ $d(f(x),f(y))\leq d(x,y)$. A map $f:Y\rightarrow Y$ is
called \textit{contractive }if for any distinct $x,y\in Y,$ $
d(f(x),f(y))<d(x,y).$

Following Beardon \cite{Be2} and Karlsson \cite{Ka1} we consider the
following properties of a metric space $(Y,d)$ that A. Karlsson called
axioms:\medskip

\noindent \textbf{Axiom 1}. The metric space $(Y,d)$ is an open dense subset
of a compact metric space $(\overline{Y},\overline{d})$, whose relative
topology coincides with the topology of $Y$. For any $w\in Y$, if $\{x_{n}\}$
is a sequence in $Y$ converging to $\xi \in \partial Y=\overline{Y}\setminus
Y$, then
\[
d(x_{n},w)\rightarrow \infty .
\]

\medskip \noindent \textbf{Axiom 2}. If $\{x_{n}\}$ and $\{y_{n}\}$ are
sequences in $Y$ converging to distinct points in $\partial Y$ then, for
every $w\in Y,$
\[
d(x_{n},y_{n})-\max \{d(x_{n},w),d(y_{n},w)\}\rightarrow \infty .
\]

\medskip \noindent \textbf{Axiom 3}. If $\{x_{n}\}$ and $\{y_{n}\}$ are
sequences in $Y$, $x_{n}\rightarrow \xi \in \partial Y$ and if for some $
w\in Y,$
\[
d(x_{n},y_{n})-d(y_{n},w)\rightarrow -\infty ,
\]
then $y_{n}\rightarrow \xi .$

\medskip \noindent \textbf{Axiom 4}. If $\{x_{n}\}$ and $\{y_{n}\}$ are
sequences in $Y$, $x_{n}\rightarrow \xi \in \partial Y$, and if for all $n,$
\[
d(x_{n},y_{n})\leq c
\]
for some constant $c$, then $y_{n}\rightarrow \xi .$ \medskip

Notice that Axiom $1$ implies that $Y$ is proper, i.e., closed bounded sets
are compact. Note that proper spaces are complete. Furthermore, Axioms $1,2$
imply\ Axiom $3$ and Axioms $1,3$ imply Axiom $4.$ For simplicity, we assume
that the compactification of $Y$ is metrizable but most results are valid
for any sequentially compact Hausdorff topological space $\overline{Y}$ (for
example, Prop. \ref{bhnonempty} is then a little narrower). We note that if $
w\in Y$, then $x_{n}\rightarrow w$ in $Y$ iff $x_{n}\rightarrow w$ in $
\overline{Y}$ but if $x_{n}\rightarrow w\in \partial Y$ (in $\overline{Y}$),
then from Axiom $1$, $d(x_{n},w)\rightarrow \infty $. For simplicity of
notation, we do not usually distinguish between these two types of
convergence.

The following theorem, proved by Ca\l ka in \cite{Ca}, is one of the
classical arguments in this line of research. In order to be self-contained
as much as possible, we present a direct proof due to Gou\"{e}zel \cite[
Lemma 2.6]{Go}.

\begin{theorem}
	\label{calka}Suppose that $(Y,d)$ is a proper metric space. Let $x_{0}\in Y$
	and $f:Y\rightarrow Y$ be a nonexpansive mapping. If there exists a sequence
	$\{n_{k}\}$ such that a subsequence $\{f^{n_{k}}(x_{0})\}$ is bounded then
	the orbit $O(x_{0})$ of $f$ is bounded.
\end{theorem}

\begin{proof}
	Fix $x_{0}\in Y$ and consider a nonexpansive mapping $f:Y\rightarrow Y$. Let
	$O(x_{0})=\{f^{n}(x_{0}),n=1,2,\ldots \}$ be the orbit of $x_{0}$. By
	assumption, since $Y$ is proper, there exists a subsequence $
	\{f^{n_{k}}(x_{0})\}$ converging to some $x$. Hence there exists $k_{0}$
	such that $f^{n_{k}}(x_{0})\in B(x,\frac{1}{2})$ for all $k\geq k_{0}$.
	Consider the set $B=\overline{O(x_{0})}\cap \bar{B}(x,1)$ and notice that $B$
	is compact since $Y$ is proper. Hence there exists a finite number of balls $
	B(x_{i},\frac{1}{2}),$ $x_{i}\in O(x_{0}),i=1,\ldots,N$ which cover $B$. Let $
	x_{i}=f^{m_{i}}(x_{0})$ for each $i$ and take $k$ large enough such that $
	n_{k}-m_{i}>0$ for $i=1,\ldots,N,$ and
	\[
	f^{n_{k}}(x_{0})=f^{n_{k}-m_{i}}(f^{m_{i}}(x_{0}))=f^{n_{k}-m_{i}}(x_{i})\in
	B\Big(x,\frac{1}{2}\Big).
	\]
	Let $i\in \{1,\ldots,N\}, \ B_{i}=\overline{O(x_{0})}\cap \bar{B}(x_{i},\frac{1}{2}
	)$ and choose $y\in B_{i}.$ Since $f$ is nonexpansive we have
	\[
	d(f^{n_{k}-m_{i}}(y),x)\leq
	d(f^{n_{k}-m_{i}}(y),f^{n_{k}-m_{i}}(x_{i}))+d(f^{n_{k}-m_{i}}(x_{i}),x)\leq
	1.
	\]
	Since $y$ is arbitrary, $f^{n_{k}-m_{i}}(B_{i})\subset B$ for all $i$. Let $
	\overline{m}=\max_{i=1,\ldots,N}\{n_{k}-m_{i}\}$. Then, for $n>\overline{m},$
	\begin{equation}
	\begin{split}
	f^{n}(B) \subset \bigcup_{i}f^{n}(B_{i})=
	\bigcup_{i}f^{n-(n_{k}-m_{i})}(f^{n_{k}-m_{i}}(B_{i})) \\
	\subset \bigcup_{i}f^{n-(n_{k}-m_{i})}(B)\subset \bigcup_{m<n}f^{m}(B).
	\end{split}
	\label{bcups}
	\end{equation}
	We now proceed by induction. We have $f^{\overline{m}+1}(B)\subset
	\bigcup_{m\leq \overline{m}}f^{m}(B).$ Fix $n>\overline{m}$ and suppose that
	$f^{j}(B)\subset \bigcup_{m\leq \overline{m}}f^{m}(B)$ for all $j\leq n$.
	Then (\ref{bcups}) yields
	\[
	f^{n+1}(B)\subset \bigcup_{j\leq n}f^{j}(B)\subset \bigcup_{m\leq \overline
		{m }}f^{m}(B).
	\]
	Thus we obtain by induction that $f^{n}(B)\subset \bigcup_{m\leq \overline{m}
	}f^{m}(B)$ for all $n>\overline{m}$. Since $\bigcup_{m\leq \overline{m}
	}f^{m}(B)$ is bounded and $f^{n_{k_{0}}}(x_{0})\in B$ it follows that $
	O(x_{0})$ is bounded too.
\end{proof}

We note that the original Ca\l ka theorem concerns a metric space $M$ with
the property that each bounded subset is totally bounded but then the
completion of $M$ is proper and the above proof is valid also in this case.
Notice that Ca\l ka's theorem implies the orbits of nonexpansive mappings
acting on a proper space are either bounded or diverge to infinity. We shall
use this observation repeatedly in different contexts.

Another tool needed for investigation of the dynamics of iterates of
nonexpansive mappings is the notion of a horoball. We recall the general
definitions introduced by Abate in \cite{Ab}. Define the \textit{small
	horoball} of center $\xi \in \partial Y$, pole $z_{0}\in Y$ and radius $r\in
\mathbb{R}$ by
\[
E_{z_{0}}(\xi ,r)=\{y\in Y:\limsup_{w\rightarrow \xi }d(y,w)-d(w,z_{0})\leq
r\}
\]
and the \textit{big horoball }by
\[
F_{z_{0}}(\xi ,r)=\{y\in Y:\liminf_{w\rightarrow \xi }d(y,w)-d(w,z_{0})\leq
r\}.
\]
Arguing as in \cite[Lemma 5.1]{Bu1} we can show that the big horoball is
always nonempty.

\begin{proposition}
	\label{bhnonempty}Suppose that $(Y,d)$ is a proper geodesic space satisfying
	Axiom $1$, $z_{0}\in Y$, $r\in \mathbb{R}$ and $\xi \in \partial Y$. Then
	the big horoball $F_{z_{0}}(\xi ,r)$ is nonempty.
\end{proposition}

\begin{proof}
	Fix $z_{0}\in Y, \ \xi \in \partial Y$ and $r>0$. It follows from Axiom $1$
	that there exists a sequence $\{y_{n}\}$ in $Y$ converging to $\xi \in
	\partial Y$ such that $d(y_{n},z_{0})\rightarrow \infty .$ We can assume
	that $d(y_{n},z_{0})>r$ for all $n$. Let $x_{n}$ be a point on the geodesic
	segment $[z_{0},y_{n}]$ such that $d(x_{n},z_{0})=r.$ Then
	\[
	d(y_{n},z_{0})=d(y_{n},x_{n})+d(x_{n},z_{0})=d(y_{n},x_{n})+r.
	\]
	Since $Y$ is proper, $\bar{B}(z_{0},r)$ is compact and hence there exists a
	subsequence $\{x_{n_{k}}\}$ of $\{x_{n}\}$ converging to some $x\in Y$ such
	that $d(x_{n_{k}},z_{0})\rightarrow d(x,z_{0})=r$ as $k\rightarrow \infty $.
	Therefore we have
	\begin{eqnarray*}
		\liminf_{w\rightarrow \xi }d(x,w)-d(w,z_{0}) &\leq &\liminf_{k\rightarrow
			\infty }d(x,y_{n_{k}})-d(y_{n_{k}},z_{0}) \\
		&=&\liminf_{k\rightarrow \infty }d(x_{n_{k}},y_{n_{k}})-d(y_{n_{k}},z_{0}) \\
		&=&-r.
	\end{eqnarray*}
	Thus we get $x\in F_{z_{0}}(\xi ,-r)$.
\end{proof}

The following theorem extends \cite[Theorem 2.3]{Ab} concerning bounded
convex domains of $\mathbb{C}^{n}$ with the Kobayashi distance to general
metric spaces.

\begin{theorem}
	\label{inv}Suppose that $(Y,d)$ satisfies Axioms $1$ and $4$ and
	$f:Y\rightarrow Y$ is a nonexpansive mapping without a bounded orbit. Then
	for every $z_{0}\in Y$ there exists $\xi \in \partial Y$ such that
	\[
	f^{k}(E_{z_{0}}(\xi ,r))\subset F_{z_{0}}(\xi ,r)
	\]
	for every $k\in \mathbb{N}$ and $r\in \mathbb{R}$.
\end{theorem}

\begin{proof}
	Fix $y,z_{0}\in Y$. It follows from Axiom $1$ that $Y$ is proper and hence,
	by Ca\l ka's theorem, $\lim_{n\rightarrow \infty }d(f^{n}(y),z_{0})=\infty $.
	By Observation 3.1 of \cite{Ka1}, there is a subsequence $\{f^{n_{i}}(y)\}$
	of $\{f^{n}(y)\}$ such that $d(f^{m}(y),z_{0})<d(f^{n_{i}}(y),z_{0})$ for
	each $m<n_{i},\ i=1,2,\ldots $. We can assume from Axiom $1$ (by passing to
	another subsequence if necessary) that $\{f^{n_{i}}(y)\}$ converges in
	$\overline{Y}$ to some point $\xi $, and $\xi \in \partial Y$ (because
	otherwise $\{d(f^{n_{i}}(y),z_{0})\}$ would be bounded). Fix $k\in \mathbb{N}$,
	$r\in \mathbb{R}$ and let $z\in E_{z_{0}}(\xi ,r)$. Then
	\begin{eqnarray*}
		\liminf_{w\rightarrow \xi }d(f^{k}(z),w)-d(w,z_{0}) &\leq
		&\liminf_{i\rightarrow \infty }d(f^{k}(z),f^{n_{i}}(y))-d(f^{n_{i}}(y),z_{0})
		\\
		&\leq &\liminf_{i\rightarrow \infty
		}d(z,f^{n_{i}-k}(y))-d(f^{n_{i}}(y),z_{0}) \\
		&\leq &\liminf_{i\rightarrow \infty
		}d(z,f^{n_{i}-k}(y))-d(f^{n_{i}-k}(y),z_{0}) \\
		&\leq &\limsup_{i\rightarrow \infty
		}d(z,f^{n_{i}-k}(y))-d(f^{n_{i}-k}(y),z_{0})
	\end{eqnarray*}%
	Since
	\[
	d(f^{n_{i}-k}(y),f^{n_{i}}(y))\leq d(y,f^{k}(y))=c
	\]%
	and $f^{n_{i}}(y)\rightarrow \xi $ we deduce from Axiom $4$ that
	$f^{n_{i}-k}(y)\rightarrow \xi $, too. Thus
	\[
	\liminf_{w\rightarrow \xi }d(f^{k}(z),w)-d(w,z_{0})\leq
	\limsup_{w\rightarrow \xi }d(z,w)-d(w,z_{0})\leq r.
	\]
\end{proof}

Theorem \ref{inv} is a very useful tool in proving Wolff--Denjoy type
theorems provided small horospheres $E_{z_{0}}(\xi ,r)$ are nonempty for
every $r\in \mathbb{R}$ (see \cite{Ab}). However, it is not clear whether it
is true in general, even for strictly convex domains in $\mathbb{C}^{n}$
with the Kobayashi distance. In the next section we show how to overcome
this difficulty.

We conclude this section by recalling the definitions of Hilbert's and
Thompson's metrics on a cone, and the Kobayashi distance. Let $K$ be a
closed normal cone with non-empty interior in a real Banach space $V$. We
say that $y\in K$ \textit{dominates} $x\in V$ if there exists $\alpha ,\beta
\in \mathbb{R}$ such that $\alpha y\leq x\leq \beta y$. This notion yields
on $K$ an equivalence relation $\sim _{K}$ by $x\sim _{k}y$ if $x$ dominates
$y$ and $y$ dominates $x$. For all $x,y\in K$ such that $x\sim _{K}y$ and
$y\neq 0,$ define
\[
M(x/y)=\inf \{\beta >0:x\leq \beta y\}
\]
and
\[
m(x/y)=\sup \{\alpha >0:\alpha y\leq x\}.
\]
The\textit{\ Hilbert (pseudo-)metric} is defined by
\[
d_{H}(x,y)=\log \bigg(\frac{M(x/y)}{m(x/y)}\bigg),
\]
and the\textit{\ Thompson metric} is
\[
d_{T}(x,y)=\log (\max \{M(x/y),M(y/x)\}).
\]
Moreover, we put $d_{H}(0,0)=d_{T}(0,0)=0$ and $d_{H}(x,y)=d_{T}(x,y)=\infty
$ if $x\nsim _{K}y$. It can be shown that $d_{T}$ is a metric on each
equivalence class of $K$, while $d_{H}(x,y)=0$ iff $x=\lambda y$ for some $
\lambda >0.$

If $\varphi \in V^{\ast }$ is a strictly positive functional, then $\Sigma
_{\varphi }^{\circ }=\{x\in \inter K:\varphi (x)=1\}$ endowed with Hilbert's
metric is a complete geodesic space whose topology coincides with the norm
topology (see, e.g., \cite{LeNu1, Th}). In finite dimensional spaces, the
set $\Sigma _{\varphi }^{\circ }$ is bounded in norm but it may be unbounded
in infinite dimensional spaces. If $D$ is a bounded open convex subset of
$V$, then there exist a closed normal cone with non-empty interior $K$ in $
V\times \mathbb{R}$ and a strictly positive functional $\varphi \in (V\times
\mathbb{R})^{\ast }$ such that $D$ is isometric to $\Sigma _{\varphi
}^{\circ }$ and thus it is equipped with Hilbert's metric $d_{H}$ (see,
e.g., \cite[p. 206]{Nu}). If $x,y \in D$, consider the straight line
passing through $x$ and $y$ that intersects $\partial D$ in
precisely two points $a$ and $b$. Assuming that $x$ is between $a$ and $y$, and $y$ is between $x$ and $b$, we define the cross-ratio
metric
\[
\widehat{d}_{H}(x,y)=\log \left( \frac{\left\Vert y-a\right\Vert \left\Vert
	x-b\right\Vert }{\left\Vert x-a\right\Vert \left\Vert
	y-b\right\Vert }\right) .
\]
It is well known that $d_{H}(x,y)=\widehat{d}_{H}(x,y)$ for every $x,y\in D.$
The following estimation will be needed in Sections 4 and 5 (see \cite{LeNu,
	Nu0}).

\begin{lemma}
	\label{h}Let $K\subset V$ be a closed normal cone with non-empty interior, $
	\varphi \in V^{\ast }$ a positive functional and $\Sigma _{\varphi }^{\circ
	}=\{x\in \inter K:\varphi (x)=1\}.$ Then $d_{H}(sx+(1-s)y,z)\leq \max
	\{d_{H}(x,z),d_{H}(y,z)\}$ for all $x,y,z\in \Sigma _{\varphi }^{\circ }$
	and $s\in \lbrack 0,1]$.
\end{lemma}

\begin{proof}
	Fix $x,y,z\in \Sigma _{\varphi }^{\circ }$ and let $r=\max
	\{d_{H}(x,z),d_{H}(y,z)\}$. By the definition of Hilbert's metric, we obtain
	that $\alpha _{1}z\leq x\leq \beta _{1}z$ and $\alpha _{2}z\leq y\leq \beta
	_{2}z$, where $0\leq \frac{\beta _{1}}{\alpha _{1}},\frac{\beta _{2}}{\alpha
		_{2}}\leq e^{r}$. If $s\in \lbrack 0,1]$ then
	\begin{equation}
	\alpha _{s}z=(s\alpha _{1}+(1-s)\alpha _{2})z\leq sx+(1-s)y\leq (s\beta
	_{1}+(1-s)\beta _{2})z=\beta _{s}z.  \label{hmmax}
	\end{equation}
	Hence
	\begin{eqnarray*}
		\tfrac{\beta _{s}}{\alpha _{s}} &=&\tfrac{s\beta _{1}+(1-s)\beta _{2}}
		{\alpha _{s}}=\tfrac{s\alpha _{1}}{\alpha _{s}}\cdot \tfrac{\beta _{1}}
		{\alpha _{1}}+\tfrac{(1-s)\alpha _{2}}{\alpha _{s}}\cdot \tfrac{\beta _{2}}
		{\alpha _{2}} \\
		&\leq &\tfrac{s\alpha _{1}}{\alpha _{s}}\cdot e^{r}+\tfrac{(1-s)\alpha _{2}}
		{\alpha _{s}}\cdot e^{r}=\tfrac{s\alpha _{1}+(1-s)\alpha _{2}}{\alpha _{s}}
		\cdot e^{r} \\
		&=&e^{\max \{d_{H}(x,z),d_{H}(y,z)\}}.
	\end{eqnarray*}
	From the above and (\ref{hmmax}) we have
	\[
	d_{H}(sx+(1-s)y,z)\leq \max \{d_{H}(x,z),d_{H}(y,z)\}.
	\]
\end{proof}

A similar argument works for Thompson's metric. The same estimation holds
also for the Kobayashi distance (see, e.g., \cite{Ab1}). Recall that if $D$
is a bounded convex domain of a complex Banach space $V$, then the \textit{Kobayashi distance} between $z,w\in D$ is given by
\[
k_{D}(z,w)=\inf \{k_{\Delta }(0,\gamma )\mid \exists \varphi \in \mbox{Hol}
(\Delta ,D):\varphi (0)=z,\,\varphi (\gamma )=w\},
\]
where $k_{\Delta }$ denotes the Poincar\'{e} metric on the unit disc $\Delta
$.

\begin{lemma}
	\label{k}Let $D\subset V$ be a bounded convex domain. Then for all $x,y,z\in
	D$ and $s\in \lbrack 0,1]$ $k_{D}(sx+(1-s)y,z)\leq \max
	\{k_{D}(x,z),k_{D}(y,z)\}$.
\end{lemma}

\begin{proof}
	Fix $x,y,z\in D$ and $s\in \lbrack 0,1]$. Without loss of generality we can
	assume that $k_{D}(x,z)>k_{D}(y,z)$. Then for every $\varepsilon >0$ there
	exist holomorphic mappings $f,g:\Delta \rightarrow D$ and $0\leq \xi <\zeta
	<1$ such that $f(0)=x$, $g(0)=y$, $f(\zeta )=g(\xi )=z,$ $k_{\Delta
	}(0,\zeta )<k_{D}(x,z)+\varepsilon $ and $k_{\Delta }(0,\xi
	)<k_{D}(y,z)+\varepsilon $. Since $D$ is convex we can define a holomorphic
	mapping $h:\Delta \rightarrow D$ by
	\[
	h(\gamma )=sf(\gamma )+(1-s)g\Big(\frac{\xi \gamma }{\zeta }\Big),\;\gamma
	\in \Delta.
	\]
	Hence
	\[
	h(0)=sf(0)+(1-s)g(0)=sx+(1-s)y
	\]
	and
	\[
	h(\zeta )=sf(\zeta )+(1-s)g(\xi)=z.
	\]
	From the above and since, as a holomorphic, the mapping $h$ is nonexpansive
	we have
	\[
	k_{D}(sx+(1-s)y,z)=k_{D}(h(0),h(\zeta ))\leq k_{\Delta }(0,\zeta
	)<k_{D}(x,z)+\varepsilon .
	\]
	Since $\varepsilon $ is arbitrary we get
	\[
	k_{D}(sx+(1-s)y,z)\leq \max \{k_{D}(x,z),k_{D}(y,z)\}.
	\]
\end{proof}

\section{Wolff--Denjoy theorems in proper spaces}

In this section, $(Y,d)$ always denotes a complete locally compact geodesic
space and we will assume that $Y$ satisfies Axiom $1.$ This involves no loss
of generality since the Hopf-Rinow theorem implies that $Y$ is proper and
hence admits a compactification satisfying Axiom $1$. From now on we denote
by $\overline{A}$ the closure of $A\subset Y$ with respect to the topology
of $(\overline{Y},\overline{d}).$

The following lemma is one of the basic tools in our consideration.

\begin{lemma}
	\label{lem1}Suppose that $f:Y\rightarrow Y$ is a nonexpansive mapping
	without a bounded orbit. Then there exists $\xi \in \partial Y$ such that
	for every $z_{0}\in Y$, $r\in \mathbb{R}$ and a sequence of natural numbers $
	\{a_{n}\}$, there exists $z\in Y$ and a subsequence $\{a_{n_{k}}\}$ of $
	\{a_{n}\}$ such that $f^{a_{n_{k}}}(z)\in F_{z_{0}}(\xi ,r)$ for every $k\in
	\mathbb{N}$. Moreover, if $Y$ satisfies Axiom $4$, then $\xi \in
	\bigcap_{r\in \mathbb{R}}\overline{F_{z_{0}}(\xi ,r)}.$
\end{lemma}

\begin{proof}
	Fix $y\in Y.$ Since $d_{n}=d(f^{n}(y),y)$ is not bounded, the Ca\l ka
	theorem \ref{calka} yields $d_{n}\rightarrow \infty $ as $n\rightarrow
	\infty $. By Observation 3.1 in \cite{Ka1}, there is a sequence $\{\varphi
	(i)\}$ of natural numbers such that $d_{m}<d_{\varphi (i)}$ for $m<\varphi
	(i), \ i=1,2,\ldots $. Since $Y$ satisfies Axiom $1$, we can assume by passing
	to another subsequence that $\{f^{\varphi (i)}(y)\}$ converges to some point
	$\xi \in \partial Y.$ Fix $r>0$ and a sequence of natural numbers $
	\{a_{n}\}. $ On each geodesic segment $[y,f^{\varphi (i)-a_{1}}(y)]\subset Y$
	(assuming that $\varphi (i)>a_{1}$ and $d(f^{\varphi (i)-a_{1}}(y),y)>r$)
	joining $y$ to $f^{\varphi (i)-a_{1}}(y)$ select a point $z_{\varphi
		(i)}^{1} $ such that $d(z_{\varphi (i)}^{1},y)=r$. Since $Y$ is proper there
	is a subsequence $\{\varphi _{1}(i)\}$ of $\{\varphi (i)\}$ such that $
	\{z_{\varphi _{1}(i)}^{1}\}$ converges to some $z_{1}\in Y.$ Hence
	\[
	\lim_{i\rightarrow \infty }d(z_{1},f^{\varphi
		_{1}(i)-a_{1}}(y))-d(f^{\varphi _{1}(i)-a_{1}}(y),y)=-d(z_{1},y)=-r.
	\]
	By induction, for every $n>1$, we can select on each geodesic segment
	$[y,f^{\varphi _{n-1}(i)-a_{n}}(y)]$ $\subset~Y$ (assuming that $\varphi
	_{n-1}(i)>a_{n}$ and $d(f^{\varphi _{n-1}(i)-a_{n}}(y),y)>r$) a point $
	z_{\varphi _{n-1}(i)}^{n}$ such that $d(z_{\varphi _{n-1}(i)}^{n},y)=r$ and
	a subsequence $\{\varphi _{n}(i)\}$ of $\{\varphi _{n-1}(i)\}$ such that $
	\{z_{\varphi _{n}(i)}^{n}\}$ converges to some $z_{n}\in Y.$ Hence
	\[
	\lim_{i\rightarrow \infty }d(z_{n},f^{\varphi
		_{n}(i)-a_{n}}(y))-d(f^{\varphi _{n}(i)-a_{n}}(y),y)=-d(z_{n},y)=-r.
	\]
	Since $Y$ is proper there exists a subsequence $\{z_{n_{k}}\}$ of $\{z_{n}\}$
	converging to some $z\in Y$ such that $d(z_{n_{k}},z)<\frac{1}{2}$ for each $
	k$. By diagonalization, there is a subsequence $\{\psi (i)\}_{i\in \mathbb{N}
	}$ of any $\{\varphi _{n}(i)\}_{i\in \mathbb{N}}$ such that
	\[
	d(z,f^{\psi (i)-a_{n_{k}}}(y))-d(f^{\psi (i)-a_{n_{k}}}(y),y)<-r+1
	\]
	for every $k\in \mathbb{N}$ and $i\geq k.$ Hence, for every $a_{n_{k}}$,
	\begin{eqnarray*}
		\liminf_{w\rightarrow \xi }d(f^{a_{n_{k}}}(z),w)-d(w,y) &\leq
		&\limsup_{i\rightarrow \infty }d(f^{a_{n_{k}}}(z),f^{\psi (i)}(y))-d(f^{\psi
			(i)}(y),y) \\
		&\leq &\limsup_{i\rightarrow \infty }d(z,f^{\psi
			(i)-a_{n_{k}}}(y))-d(f^{\psi (i)}(y),y) \\
		&\leq &\limsup_{i\rightarrow \infty }d(z,f^{\psi
			(i)-a_{n_{k}}}(y))-d(f^{\psi (i)-a_{n_{k}}}(y),y) \\
		&\leq &-r+1,
	\end{eqnarray*}
	and thus
	\[
	\liminf_{w\rightarrow \xi }d(f^{a_{n_{k}}}(z),w)-d(w,z_{0})\leq
	-r+1+d(z_{0},y)
	\]
	for every $z_{0}\in Y$. Since $r>0$ is arbitrary, this proves the first part
	of the lemma. To prove that $\xi \in \bigcap_{r\in \mathbb{R}}\overline{
		F_{z_{0}}(\xi ,r)}$ notice that (taking $a_{i}=\varphi (i)$) for every $
	z_{0}\in Y$ and $r\in \mathbb{R}$ there is $\tilde{z}\in Y$ and a
	subsequence $\{\varphi (i_{k})\}$ of $\{\varphi (i)\}$ such that $f^{\varphi
		(i_{k})}(\tilde{z})\in F_{z_{0}}(\xi ,r)$ for each $k$. Since
	\[
	d(f^{\varphi (i_{k})}(y),f^{\varphi (i_{k})}(\tilde{z}))\leq d(y,\tilde{z}),
	\]
	it follows from Axiom $4$ that $f^{\varphi (i_{k})}(\tilde{z})\rightarrow
	\xi .$ Thus $\xi \in \overline{F_{z_{0}}(\xi ,r)}$ and from arbitrariness of
	$r$, $\xi \in \bigcap_{r\in \mathbb{R}}\overline{F_{z_{0}}(\xi ,r)}.$
\end{proof}

We are now in a position to prove a general Wolff--Denjoy type theorem in
proper geodesic spaces.

\begin{theorem}
	\label{main}Let $(Y,d)$ satisfy Axiom $4$ and suppose that for every $\zeta
	\in \partial Y$ and $z_{0}\in Y$ the intersection of horoballs' closures $
	\bigcap_{r\in \mathbb{R}}\overline{F_{z_{0}}(\zeta ,r)}$ consists of a
	single point. If $f:Y\rightarrow Y$ is a nonexpansive mapping without
	bounded orbits, then there exists $\xi \in \partial Y$ such that the
	iterates $f^{n}$ of $f$ converge uniformly on bounded sets of $Y$ to $\xi .$
\end{theorem}

\begin{proof}
	By assumption, there exists $\xi \in \partial Y$ that satisfies the
	conclusion of Lemma \ref{lem1}. Hence $\bigcap_{r\in \mathbb{R}}\overline{
		F_{z_{0}}(\xi ,r)}=\{\xi \}.$ Fix $y\in Y$ and select a subsequence $
	\{f^{a_{n}}(y)\}$ of iterates of $f$ converging to some $\eta \in \partial Y$. 
	We show that $\eta =\xi .$ Fix $r\in \mathbb{R}$. From Lemma \ref{lem1}
	there exists a subsequence $\{a_{n_{k}}\}$ of $\{a_{n}\}$ and $z\in Y$ such
	that
	\begin{equation}
	f^{a_{n_{k}}}(z)\in F_{z_{0}}(\xi ,r)\mbox{ for }k\in \mathbb{N}.
	\label{bhl1}
	\end{equation}
	Since $d(f^{a_{n_{k}}}(y),f^{a_{n_{k}}}(z))\leq d(y,z),$ we obtain from
	Axiom $4$ that $\{f^{a_{n_{k}}}(z)\}$ converges to $\eta $ too. From the
	above and (\ref{bhl1}) we have $\eta \in \overline{F_{z_{0}}(\xi ,r)}$ for
	all $r\in \mathbb{R}$. Hence
	\[
	\eta \in \bigcap_{r\in \mathbb{R}}\overline{F_{z_{0}}(\xi ,r)}=\{\xi \}.
	\]
	Therefore $\eta =\xi $ and $f^{a_{n}}(y)\rightarrow \xi $ for any converging
	subsequence $\{f^{a_{n}}(y)\}$. It follows that $f^{n}(y)\rightarrow \xi $
	for each $y\in Y.$ The proof of uniform convergence on bounded sets is
	routine: suppose, on the contrary, that there exist an open neighbourhood $
	U\subset \overline{Y}$ of $\xi $, a bounded set $K\subset Y$ and a sequence $
	\{y_{n}\}\subset K$ such that $f^{n}(y_{n})\notin U$ for each $n$. Then
	\[
	d(f^{n}(y_{n}),f^{n}(y))\leq d(y_{n},y)\leq \diam K
	\]
	for any $y\in K$ and, since $f^{n}(y)\rightarrow \xi $, we deduce from Axiom
	4 that $f^{n}(y_{n})\rightarrow \xi \in \overline{Y}\setminus U$, a
	contradiction.
\end{proof}

\begin{corollary}
	Let $(Y,d)$ satisfy Axiom $4$ and suppose that for every $\zeta \in \partial
	Y$ and $z_{0}\in Y$ the intersection $\bigcap_{r\in \mathbb{R}}\overline{
		F_{z_{0}}(\zeta ,r)}$ consists of a single point. If $f:Y\rightarrow Y$ is a
	contractive mapping then there exists $\xi \in \overline{Y}$ such that the
	iterates $f^{n}$ of $f$ converge uniformly on bounded sets of $Y$ to $\xi $.
\end{corollary}

\begin{proof}
	If $f:Y\rightarrow Y$ has unbounded orbits then the conclusion follows
	directly from Theorem \ref{main}. Therefore, we suppose that $\{f^{n}(y)\}$
	is bounded for all $y\in Y$. Fix $y\in Y$. There exists a subsequence $
	\{f^{n_{i}}(y)\}$ of $\{f^{n}(y)\}$ converging to some $z\in Y$. Since $f$
	is contractive then the sequence $d_{n}=d(f^{n}(y),f^{n+1}(y))$, $
	n=1,2,\ldots $, is decreasing and hence it converges to some $\eta $ as $
	n\rightarrow \infty $. Note by continuity that $\eta
	=d(z,f(z))=d(f(z),f^{2}(z))$. If $z$ and $f(z)$ were distinct points we
	would have received a contradiction (with $f$ being contractive). Hence $
	f(z)=z$ and, since $\{d(f^{n}(y),z)\}$ is decreasing, we obtain that $
	f^{n}(y)\rightarrow z\in Y$ if $n\rightarrow \infty $ for every $y\in Y$.
\end{proof}

It is not difficult to see that Beardon's Axiom 2 implies the property that
any horoball's closure $\overline{F_{z_{0}}(\xi ,r)}$ meets the boundary $
\partial Y$ at only one point $\xi $. The following lemma, combined with
Theorem \ref{main} shows that in proper geodesic spaces this is more than
required to prove a Wolff--Denjoy type theorem.

\begin{lemma}
	\label{a3}Let $(Y,d)$ satisfy Axiom $3$, $z_{0}\in Y$, $\zeta \in \partial Y$
	and $r\in \mathbb{R}$. Then $\bigcap_{r\in \mathbb{R}}\overline{
		F_{z_{0}}(\zeta ,r)}=\{\zeta \}$.
\end{lemma}

\begin{proof}
	Fix $z_{0}\in Y$ and $\zeta \in \partial Y$. If follows from Proposition
	\ref{bhnonempty} that a big horoball is nonempty and hence by compactness, $
	\bigcap_{r\in \mathbb{R}}\overline{F_{z_{0}}(\zeta ,r)}\neq \emptyset $.
	Suppose that $\eta \in \bigcap_{r\in \mathbb{R}}\overline{F_{z_{0}}(\zeta
		,r) }\subset \partial Y$ and consider a sequence $\{r_{n}\}$ of real numbers
	diverging to $-\infty $. Hence, for each $n\in \mathbb{N}$, $\eta \in
	\overline{F_{z_{0}}(\zeta ,r_{n})}$ and thus there exists a sequence $
	\{z_{i}^{r_{n}}\}_{i\in \mathbb{N}}$ contained in $F_{z_{0}}(\zeta ,r_{n})$
	such that $\overline{d}(z_{i}^{r_{n}},\eta )\rightarrow 0$ if $i\rightarrow
	\infty $. It follows that for every $n\in \mathbb{N}$ there exists $i_{n}\in
	\mathbb{N}$ such that $\overline{d}(z_{i_{n}}^{r_{n}},\eta )<\frac{1}{n}$
	and
	\[
	\liminf_{w\rightarrow \zeta }d(z_{i_{n}}^{r_{n}},w)-d(w,z_{0})\leq r_{n}.
	\]
	Let $\{w_{k}^{n}\}_{k\in \mathbb{N}}$ be a sequence converging to $\zeta \in
	\partial Y$ such that
	\[
	\lim_{k\rightarrow \infty
	}d(z_{i_{n}}^{r_{n}},w_{k}^{n})-d(w_{k}^{n},z_{0})\leq r_{n}.
	\]
	Then, for each $n\in \mathbb{N}$, there exists $w_{k_{n}}^{n}$ such that $
	\overline{d}(w_{k_{n}}^{n},\zeta )\leq \frac{1}{n}$ and
	\[
	d(z_{i_{n}}^{r_{n}},w_{k_{n}}^{n})-d(w_{k_{n}}^{n},z_{0})\leq r_{n}+\frac{1}{
		n}.
	\]
	Hence we obtain
	\[
	\lim_{n\rightarrow \infty
	}d(z_{i_{n}}^{r_{n}},w_{k_{n}}^{n})-d(w_{k_{n}}^{n},z_{0})=-\infty .
	\]
	By Axiom $3$, $\overline{d}(w_{k_{n}}^{n},\eta )\rightarrow 0$ but $
	\overline{d}(w_{k_{n}}^{n},\zeta )\rightarrow 0$, too. This means that $\eta
	=\zeta $ and therefore, $\bigcap_{r\in \mathbb{R}}\overline{F_{z_{0}}(\zeta
		,r)}=\{\zeta \}.$
\end{proof}

Thus Axiom 3 can be regarded as an abstract formulation of the property that
the intersection of closures of horoballs consists of a single point.
Theorem \ref{main} and Lemma \ref{a3} immediately lead to the following
results.

\begin{theorem}
	\label{main2}Let $(Y,d)$ satisfy Axiom $3$. If $f:Y\rightarrow Y$ is a
	nonexpansive mapping without bounded orbits, then there exists $\xi \in
	\partial Y$ such that the iterates $f^{n}$ of $f$ converge uniformly on
	bounded sets of $Y$ to $\xi .$
\end{theorem}

\begin{corollary}
	Let $(Y,d)$ satisfy Axiom $3$. If $f:Y\rightarrow Y$ is a contractive
	mapping then there exists $\xi \in \overline{Y}$ such that the iterates $
	f^{n}$ of $f$ converge uniformly on bounded sets of $Y$ to $\xi $.
\end{corollary}

\section{Strictly convex domains}

Let $V$ be a finite-dimensional (real or complex) vector space, $D$ a
bounded open subset of $V$ and $\partial D=\overline{D}\setminus D$, where $
\overline{D}$ denotes the closure of $D$ in the norm topology. To simplify
notation, we write in this section
\[
\lbrack z,w]=\{sz+(1-s)w:s\in \lbrack 0,1]\}\mbox{ \ and \ }
(z,w)=\{sz+(1-s)w:s\in (0,1)\}
\]
for the closed and open segments connecting $z$ and $w$ in $D$. We begin
with the following simple lemma (see, e.g., \cite[Lemma 1]{AbRa}).

\begin{lemma}
	\label{opensegment}Let $D\subset V$ be a convex domain. Then
	
	\begin{enumerate}
		\item[(i)] $(z,w)\subset D$ for all $z\in D$ and $w\in \partial D$
		
		\item[(ii)] if $z,w\in \partial D$, then either $(z,w)\subset \partial D$ or
		$(z,w)\subset D.$
	\end{enumerate}
\end{lemma}

\begin{proof}
	~
	
	\begin{enumerate}
		\item[(i)] Let $z\in D,w\in \partial D$ and pick $s\in (0,1).$ We show that $
		sz+(1-s)w\in D.$ Since $D$ is open and $z\in D$, there exists a ball $
		B(z,r)\subset D.$ Let $r^{\prime }<r\frac{s}{1-s}$ and select $w^{\prime
		}\in D$ such that $\left\Vert w-w^{\prime }\right\Vert <r^{\prime }.$ Then
		\begin{equation}
		sz+(1-s)w=sz^{\prime }+(1-s)w^{\prime },  \label{comb}
		\end{equation}
		where
		\[
		z^{\prime }=z+\frac{1-s}{s}(w-w^{\prime }).
		\]
		Hence $\left\Vert z-z^{\prime }\right\Vert <\frac{1-s}{s}r^{\prime }<r$ and
		therefore $z^{\prime }\in B(z,r)\subset D.$ It folows from (\ref{comb}) and
		convexity of $D$ that $sw+(1-s)z\in D.$
		
		\item[(ii)] Suppose that $(z,w)\varsubsetneq \partial D$ so there exists $
		y\in (z,w)\cap D$. Then from (i) $(z,y)\subset D$ and $(y,w)\subset D$, and
		consequently $(z,w)\subset D.$
	\end{enumerate}
\end{proof}

Recall that $D\subset V$ is \textit{strictly convex} if for any $z,w\in \overline{D}$
the open segment $(z,w)$ lies in $D$.

The objective of this section is to apply the general results of Section 3
to the case of a bounded strictly convex domain $D\subset V$. In what
follows, we will always assume that $(D,d)$ is a metric space whose topology
coincides with the norm topology. Our next lemma, although formulated for
subsets of a finite dimensional space, is valid for any proper space.

\begin{lemma}
	\label{A1}Suppose that $D$ is a bounded domain of $V$ and $(D,d)$ is a
	complete geodesic space. Then $(D,d)$ satisfies Axiom $1.$
\end{lemma}

\begin{proof}
	Suppose that a sequence $\{x_{n}\}$ converges to $\xi \in \partial D$. We
	show that $d(x_{n},y)\rightarrow \infty $ for every $y\in D.$ On the
	contrary, suppose that there exists a subsequence $\{x_{n_{k}}\}$ such that $
	d(x_{n_{k}},y)\leq c$ for some $y\in D$ and $c>0.$ Since $(D,d)$ is locally
	compact, it follows from the Hopf--Rinow theorem that $D$ is proper and thus
	there exists a subsequence $\{x_{n_{k_{l}}}\}$ that converges in $(D,d)$ to
	some $x_{0}\in D.$ But the topology of $(D,d)$ coincides with the norm
	topology and hence $\{x_{n_{k_{l}}}\}$ tends to $x_{0}\in D$ in norm, a
	contradiction (since $D$ is open and hence $\xi =x_{0}\notin \partial D$).
\end{proof}

The following proposition is valid for any metric $d$ satisfying the
following condition--equivalent to the convexity of balls in $D$:
\begin{equation}
d(sx+(1-s)y,z)\leq \max \{d(x,z),d(y,z)\}\mbox{ for all }x,y,z\in D\mbox{
	and }s\in \lbrack 0,1].
\tag{C}
\label{C}
\end{equation}

\begin{proposition}
	\label{A3star}Suppose that $D$ is a convex domain of $V$ and $(D,d)$
	satisfies condition (\ref{C}). If $\{x_{n}\}$ and $\{y_{n}\}$ are sequences
	in $D\,$converging to $\xi $ and $\eta $, respectively, in $\partial D$ and
	if for some $w\in D,$
	\[
	d(x_{n},y_{n})-d(y_{n},w)\rightarrow -\infty ,
	\]
	then $[\xi ,\eta ]\subset \partial D.$
\end{proposition}

\begin{proof}
	Let $x_{n}\rightarrow \xi ,\ y_{n}\rightarrow \eta ,\ \xi ,\eta \in \partial D$
	and
	\[
	d(x_{n},y_{n})-d(y_{n},w)\rightarrow -\infty
	\]
	for some $w\in D.$ On the contrary, suppose that $s\xi +(1-s)\eta \in D$ for
	some $s\in (0,1).$ Then
	\[
	d(sx_{n}+(1-s)y_{n},w)\geq d(w,y_{n})-d(y_{n},sx_{n}+(1-s)y_{n})\geq
	d(w,y_{n})-d(y_{n},x_{n})\rightarrow +\infty .
	\]
	Since
	\[
	||sx_{n}+(1-s)y_{n}-(s\xi +(1-s)\eta )||\leq s||x_{n}-\xi
	||+(1-s)||y_{n}-\eta ||\rightarrow 0
	\]
	and topologies of $(D,d)$ and $(\overline{D},\left\Vert \cdot \right\Vert )$
	coincide on $D$, we have
	\[
	d(s\xi +(1-s)\eta ,w)\rightarrow +\infty,
	\]
	a contradiction.
\end{proof}

The following lemma is an immediate consequence of Proposition \ref{A3star}.

\begin{lemma}
	\label{A3}Suppose that $D$ is a bounded strictly convex domain of $V$ and $
	(D,d)$ satisfies condition (\ref{C}). If $\{x_{n}\}$ and $\{y_{n}\}$ are
	sequences in $Y$, $x_{n}\rightarrow \xi \in \partial D$ and if for some $
	w\in D,$
	\[
	d(x_{n},y_{n})-d(y_{n},w)\rightarrow -\infty ,
	\]
	then $y_{n}\rightarrow \xi .$ In other words $(D,d)$ satisfies Axiom $3$.
\end{lemma}

Combining Theorem \ref{main2}, Lemmas \ref{A1}, \ref{A3} and a classical
argument from metric fixed point theory we obtain the main result of this
section. We point out that there is no need to assume any hyperbolic
property of a metric.

\begin{theorem}
	\label{main_Rn}Suppose that $D$ is a bounded strictly convex domain of $V$
	and $(D,d)$ is a complete geodesic space satisfying condition (\ref{C})
	whose topology coincides with the norm topology. If $f:D\rightarrow D$ is a
	nonexpansive mapping without fixed points, then there exists $\xi \in
	\partial D$ such that the iterates $f^{n}$ of $f$ converge uniformly on
	bounded sets of $D$ to $\xi .$
\end{theorem}

\begin{proof}
	It follows from Lemma \ref{A3} that $(D,d)$ satisfies Axiom $3$. If $
	f:D\rightarrow D$ has unbounded orbits then the conclusion follows directly
	from Lemma \ref{A1} and Theorem \ref{main2}. Therefore, we can assume that $
	\{f^{n}(y)\}$ is bounded for some (and hence for any) $y\in D.$ Let $
	r=\inf_{z\in D}\limsup_{n\rightarrow \infty }d(z,f^{n}(y))$ denote the
	asymptotic radius of $\{f^{n}(y)\}.$ Then, by properness of $(D,d)$, the
	asymptotic center
	\[
	A=\{x\in D:\limsup_{n\rightarrow \infty }d(x,f^{n}(y))=r\}
	\]
	of $\{f^{n}(y)\}$ is a nonempty compact subset of $D$ that is invariant
	under $f$, that is, $f(D)\subset D.$ Moreover $A$ is convex since $D$ is
	convex and $(D,d)$ satisfies condition (\ref{C}). From Brouwer's theorem $f$
	has a fixed point which contradicts our assumption.
\end{proof}

As in Section 3, we have also a variant of the above theorem for contractive
maps.

\begin{theorem}
	Suppose that $D$ is a bounded strictly convex domain of $V$ and $(D,d)$ is a
	complete geodesic space satisfying condition (\ref{C}) whose topology
	coincides with the norm topology. If $f:D\rightarrow D$ is a contractive
	mapping then there exists $\xi \in \overline{D}$ such that the iterates $
	f^{n}$ of $f$ converge uniformly on bounded sets of $D$ to $\xi .$
\end{theorem}

In particular, Theorem \ref{main_Rn} is valid for classical metrics
discussed in Section 2 (see Lemmas \ref{h} and \ref{k}).

\begin{corollary}
	\label{hilb}Let $D\subset \mathbb{R}^{n}$ be a bounded strictly convex
	domain. If $f:(D,d_{H})\rightarrow (D,d_{H})$ is a fixed-point free
	nonexpansive map then there exists $\xi \in \partial D$ such that the
	iterates $f^{n}$ of $f$ converge uniformly on bounded sets of $D$ to $\xi .$
\end{corollary}

\begin{corollary}
	\label{thom}Let $D\subset \mathbb{R}^{n}$ be a bounded strictly convex
	domain. If $(D,d_{T})$ is a geodesic space and $f:(D,d_{T})\rightarrow
	(D,d_{T})$ is a fixed-point free nonexpansive map then there exists $\xi \in
	\partial D$ such that the iterates $f^{n}$ of $f$ converge uniformly on
	bounded sets of $D$ to $\xi .$
\end{corollary}

Corollary \ref{hilb} was shown by Beardon \cite[Theorems 1, 1a]{Be2} who
proved a variant of the classical intersecting chord theorem and then showed
that bounded strictly convex domains $(D,d_{H})$ satisfy Axiom $2$.
Corollary \ref{thom} appears to be new, compare also
\cite[Theorem 3.2]{LLNW}.

A similar conclusion holds true for bounded strictly convex domains in $
\mathbb{C}^{n}$ with the Kobayashi metric.

\begin{corollary}
	\label{kob}Let $D\subset \mathbb{C}^{n}$ be a bounded strictly convex
	domain. If $f:(D,k_{D})\rightarrow (D,k_{D})$ is a fixed-point free nonexpansive map
	then there exists $\xi \in \partial D$ such that the iterates $f^{n}$ of $f$
	converge uniformly on bounded sets of $D$ to $\xi.$
\end{corollary}

Corollary \ref{kob} was shown by Budzy\'{n}ska \cite[Theorem 5.3]{Bu1} with
the use of the Earle--Hamilton theorem and properties of horoballs.

The above results leave some room for improvements. Let $D\subset V$ be a
convex domain. Given $\xi \in \partial D,F\subset \partial D$, set
\begin{eqnarray*}
	\ch (\xi ) &=&\{x\in \partial D:[x,\xi ]\subset \partial D\} \\
	\ch (F) &=&\bigcup_{\xi \in F} \ch (\xi ).
\end{eqnarray*}
If $f:D\rightarrow D, \ y\in D$ then $\omega _{f}(y)$ denotes the set of
accumulation points of $\{f^{n}(y)\}$ (in the norm topology) and $\Omega
_{f}=\bigcup_{y\in D}\omega _{f}(y)$ is the attractor of $f.$ We conclude
this section with a general version of Abate and Raissy's Theorem 6 in
\cite{AbRa}, who proved it for bounded convex domains with the Kobayashi metric.

\begin{proposition}
	Let $D$ be a bounded convex domain of $V$ and let $(D,d)$ be a complete
	geodesic space satisfying condition (\ref{C}) whose topology coincides with
	the norm topology. If $f:D\rightarrow D$ is a fixed-point free nonexpansive
	map then there exists $\xi \in \partial D$ such that
	\[
	\Omega _{f}\subset \bigcap_{r\in \mathbb{R}} \ch (\overline{F_{z_{0}}(\xi ,r)%
	}\cap \partial D)
	\]
	for some $z_{0}\in D.$
\end{proposition}

\begin{proof}
	Fix $y\in Y.$ There is no loss of generality in assuming that $d(f^{n}(y),y)$
	diverges to $\infty $ since otherwise, there is a fixed point of $f$ as in
	the proof of Theorem \ref{main_Rn}. Choose $\xi \in \partial D$ satisfying
	the conclusion of Lemma \ref{lem1}. Fix $z_{0},y\in D, \ r\in \mathbb{R}$ and
	select a subsequence $\{f^{a_{n}}(y)\}$ of iterates of $f$ converging to
	some $\eta .$ From Lemma \ref{lem1} there exists $z\in D$ and a subsequence $
	\{a_{n_{k}}\}$ of $\{a_{n}\}$ such that $f^{a_{n_{k}}}(z)\in F_{z_{0}}(\xi
	,r)$ for every $k\in \mathbb{N}.$ We can assume by passing to another
	subsequence that $f^{a_{n_{k}}}(z)\rightarrow \zeta \in \overline{
		F_{z_{0}}(\xi ,r)}.$ Clearly, $\eta ,\zeta \in \partial D$ since the orbits
	of $f$ diverge to $\infty $. Now Proposition \ref{A3star} yields $[\eta
	,\zeta ]\subset \partial D$, that is, $\eta \in \ch (\overline{F_{z_{0}}(\xi
		,r)}\cap \partial D).$ Since $r$ is arbitrary, this proves the theorem.
\end{proof}

Karlsson and Noskov showed in \cite{KaNo} that if a bounded convex domain $D$
is endowed with the Hilbert metric $d_{H}$, then the attractor $\Omega _{f}$
of a fixed-point free nonexpansive map $f:D\rightarrow D$ is a star-shaped
subset of $\partial D$. This has led to a conjecture formulated by Karlsson
and Nussbaum asserting that $\co \Omega _{f}\subset \partial D$
(see \cite{Ka3, Nu}). It remains one of the major problems in the field.

\section{The case of compact mappings}

The objective of this section is to extend the results of Section 4 to
infinite dimensional spaces. Therefore, we have to modify Axiom $1$ and in
this section we will assume\medskip

\noindent \textbf{Axiom 1'}. The metric space $(Y,d)$ is an open dense
subset of a metric space $(\overline{Y},\overline{d})$, whose relative
topology coincides with the metric topology. For any $w\in Y$, if $\{x_{n}\}$
is a sequence in $Y$ converging to $\xi \in \partial Y=\overline{Y}\setminus
Y$, then $d(x_{n},w)\rightarrow \infty $ (the compactness of $\overline{Y}$
is not required).\medskip

Notice that Axiom 1' implies that if $A\subset Y$ is bounded, then the $
\overline{d}$-closure of $A$ does not intersect the boundary $\partial Y$
and hence coincides with the $d$-closure of $A$.

Since in general, there are sequences without convergent subsequences, we
also need a slight modification of Axioms $3$ and $4$:\medskip

\noindent \textbf{Axiom 3'}. If $\{x_{n}\}$ and $\{y_{n}\}$ are sequences in
$Y$, $x_{n}\rightarrow \xi \in \partial Y$, $y_{n}\rightarrow \eta \in
\partial Y$, and if for some $w\in Y,$
\[
d(x_{n},y_{n})-d(y_{n},w)\rightarrow -\infty ,
\]
then $\xi =\eta .$\medskip

\noindent \textbf{Axiom 4'}. If $\{x_{n}\}$ and $\{y_{n}\}$ are sequences in
$Y$, $x_{n}\rightarrow \xi \in \partial Y$, $y_{n}\rightarrow \eta \in
\partial Y$, and if for all $n,$
\[
d(x_{n},y_{n})\leq c
\]
for some constant $c$, then $\xi =\eta .$\medskip

We will also consider a wider class of quasi-geodesic metric spaces. Recall
that a curve $\gamma :[a,b]\rightarrow Y$ is called 
\textit{$(\lambda ,\kappa)$-quasi-geodesic} if there exists $\lambda \geq 1,\kappa \geq 0$ such
that for all $t_{1},t_{2}\in \lbrack a,b],$
\[
\frac{1}{\lambda }\left\vert t_{1}-t_{2}\right\vert -\kappa \leq d(\gamma
(t_{1}),\gamma (t_{2}))\leq \lambda \left\vert t_{1}-t_{2}\right\vert
+\kappa .
\]
A metric space $Y$ is called $(\lambda ,\kappa)$\textit{-quasi-geodesic} if
every pair of points in $Y$ can be connected by a $(\lambda ,\kappa)$-quasi-geodesic. We say 
that a mapping $f:Y\rightarrow Y$ is \textit{compact}
if $\overline{f(Y)}$, the $\overline{d}$-closure of $f(Y)$, is compact in $(
\overline{Y},\overline{d}).$ We will need the following counterpart of Lemma
\ref{lem1}. The proof is a little more subtle.

\begin{lemma}
	\label{lem1_comp}Suppose that $Y$ is a $(1,\kappa )$-quasi-geodesic space
	satisfying Axiom 1' and $f:Y\rightarrow Y$ is a compact nonexpansive mapping
	without a bounded orbit. Then there exists $\xi \in \partial Y$ such that
	for every $z_{0}\in Y$, $r\in \mathbb{R}$ and a sequence of natural numbers $
	\{a_{n}\}$, there exists $z\in Y$ and a subsequence $\{a_{n_{k}}\}$ of $
	\{a_{n}\}$ such that $f^{a_{n_{k}}}(z)\in F_{z_{0}}(\xi ,r)$ for every $k\in
	\mathbb{N}$. Moreover, if $Y$ satisfies Axiom 4', then $\xi \in
	\bigcap_{r\in \mathbb{R}}\overline{F_{z_{0}}(\xi ,r)}.$
\end{lemma}

\begin{proof}
	Fix $y\in Y.$ Since $f:Y\rightarrow Y$ is compact, the $\overline{d}$-closure
	$\overline{O(y)}$ of the orbit $\{f^{n}(y):n\geq 1\}$ is compact in
	$\overline{Y}$ and hence, from Axiom 1', $\overline{O(y)}\cap B=O(y)\cap
	B\subset Y$ is compact in $Y$ for any $d$-closed and bounded set $B\subset Y$.
	Thus $(O(y),d)$ is proper and from Ca\l ka's theorem and nonexpansivity of
	$f$ we conclude that $d_{n}=d(f^{n}(y),y)\rightarrow \infty $ as $
	n\rightarrow \infty $. As in Lemma \ref{lem1}, there is a sequence $
	\{\varphi (i)\}$ of natural numbers such that $d_{m}<d_{\varphi (i)}$ for $
	m<\varphi (i), \ i=1,2,\ldots $. Since $f$ is compact, we can assume without
	loss of generality that $\{f^{\varphi (i)}(y)\}$ converges to some point $
	\xi \in \partial Y$. Fix $r>0$ and a sequence of natural numbers $\{a_{n}\}.$
	On each $(1,\kappa )$-quasi-geodesic segment $[y,f^{\varphi
		(i)-a_{1}-1}(y)]\subset Y$ (assuming that $\varphi (i)>a_{1}+1$ and $
	d(f^{\varphi (i)-a_{1}-1}(y),y)>r$) joining $y$ to $f^{\varphi
		(i)-a_{1}-1}(y)$ select a point $z_{\varphi (i)}^{1}$ such that $
	d(z_{\varphi (i)}^{1},y)=r$. Since $Y$ is $(1,\kappa )$-quasi-geodesic we
	have
	\[
	d(f^{\varphi(i)-a_{1}-1}(y),z_{\varphi (i)}^{1})+d(z_{\varphi
		(i)}^{1},y)\leq d(f^{\varphi(i)-a_{1}-1}(y),y)+3\kappa
	\]
	for each $i.$ By assumption, $f(Y)$ is compact in $(\overline{Y},\overline{d})$ and
	hence there is a subsequence $\{\varphi _{1}(i)\}$ of $\{\varphi (i)\}$ such
	that $\{f(z_{\varphi _{1}(i)}^{1})\}$ converges to some $z_{1}\in f(Y).$
	Hence
	\begin{eqnarray*}
		&&\limsup_{i\rightarrow \infty }d(z_{1},f^{\varphi
			_{1}(i)-a_{1}}(y))-d(f^{\varphi _{1}(i)-a_{1}-1}(y),y) \\
		&=&\limsup_{i\rightarrow \infty }d(f(z_{\varphi _{1}(i)}^{1}),f^{\varphi
			_{1}(i)-a_{1}}(y))-d(f^{\varphi _{1}(i)-a_{1}-1}(y),y) \\
		&\leq &\limsup_{i\rightarrow \infty }d(z_{\varphi _{1}(i)}^{1},f^{\varphi
			_{1}(i)-a_{1}-1}(y))-d(f^{\varphi _{1}(i)-a_{1}-1}(y),y) \\
		&\leq &\limsup_{i \to \infty} -d(z_{\varphi
			(i)}^{1},y)+3\kappa =-r+3\kappa.
	\end{eqnarray*}
	By induction, for every $n>1$, we can select on each $(1,\kappa )$-quasi-geodesic 
	segment $[y,f^{\varphi _{n-1}(i)-a_{n}-1}(y)]\subset Y$
	(assuming that $\varphi _{n-1}(i)>a_{n}-1$ and $d(f^{\varphi
		_{n-1}(i)-a_{n}-1}(y),y)>r$) a point $z_{\varphi _{n-1}(i)}^{n}$ such that $
	d(z_{\varphi _{n-1}(i)}^{n},y)=r$ and a subsequence $\{\varphi _{n}(i)\}$ of
	$\{\varphi _{n-1}(i)\}$ such that $\{f(z_{\varphi _{n}(i)}^{n})\}$ converges
	to some $z_{n}\in f(Y).$ Hence
	\[
	\limsup_{i\rightarrow \infty }d(z_{n},f^{\varphi
		_{n}(i)-a_{n}}(y))-d(f^{\varphi _{n}(i)-a_{n}-1}(y),y)\leq -r+3\kappa
	\]
	and since $\{z_{n}\}$ is contained in the compact set $\overline{f(Y)}\cap
	\bar{B}(y,r)\subset Y$ there is a subsequence $\{z_{n_{k}}\}$ of $\{z_{n}\}$
	converging to some $z\in Y$ with $d(z_{n_{k}},y)<\frac{1}{2}$ for each $k$.
	By diagonalization, there is a subsequence $\{\psi (i)\}_{i\in \mathbb{N}}$
	of any $\{\varphi _{n}(i)\}_{i\in \mathbb{N}}$ such that
	\[
	d(z,f^{\psi (i)-a_{n_{k}}}(y))-d(f^{\psi (i)-a_{n_{k}}-1}(y),y)<-r+1+3\kappa
	\]
	for every $k\in \mathbb{N}$ and $i\geq k.$ Hence, for every $a_{n_{k}}$,
	\begin{eqnarray*}
		\liminf_{w\rightarrow \xi }d(f^{a_{n_{k}}}(z),w)-d(w,y) &\leq
		&\limsup_{i\rightarrow \infty }d(f^{a_{n_{k}}}(z),f^{\psi (i)}(y))-d(f^{\psi
			(i)}(y),y) \\
		&\leq &\limsup_{i\rightarrow \infty }d(z,f^{\psi
			(i)-a_{n_{k}}}(y))-d(f^{\psi (i)}(y),y) \\
		&\leq &\limsup_{i\rightarrow \infty }d(z,f^{\psi
			(i)-a_{n_{k}}}(y))-d(f^{\psi (i)-a_{n_{k}}-1}(y),y) \\
		&\leq &-r+1+3\kappa
	\end{eqnarray*}
	
	and thus
	\[
	\liminf_{w\rightarrow \xi }d(f^{a_{n_{k}}}(z),w)-d(w,z_{0})\leq -r+1+3\kappa
	+d(z_{0},y)
	\]
	for every $z_{0}\in Y$. Since $r>0$ is arbitrary, this proves the first part
	of the lemma. As in the proof of Lemma \ref{lem1}, we show that $\xi \in
	\bigcap_{r\in \mathbb{R}}\overline{F_{z_{0}}(\xi ,r)}$.
\end{proof}

Having Lemma \ref{lem1_comp} in hand, we can prove counterparts of the
previous results.

\begin{theorem}
	\label{main_comp}Let $(Y,d)$ be a $(1,\kappa )$-quasi-geodesic space
	satisfying Axioms 1' and 3'. If $f:Y\rightarrow Y$ is a compact nonexpansive
	mapping without bounded orbits, then there exists $\xi \in \partial Y$ such
	that the iterates $f^{n}$ of $f$ converge uniformly on bounded sets of $Y$
	to $\xi .$
\end{theorem}

\begin{proof}
	By assumption, there exists $\xi \in \partial Y$ satisfying the conclusion
	of Lemma \ref{lem1_comp}. Hence $\xi \in \bigcap_{r\in \mathbb{R}}\overline{
		F_{z_{0}}(\xi ,r)}$ and arguing as in the proof of Lemma \ref{a3}, $
	\bigcap_{r\in \mathbb{R}}\overline{F_{z_{0}}(\xi ,r)}=\{\xi \}.$ Fix $y\in
	Y. $ Since $f$ is compact there exists a subsequence $\{f^{a_{n}}(y)\}$ of
	iterates of $f$ converging to some $\eta \in \partial Y$. We show that $\eta
	=\xi $ for any converging subsequence $\{f^{a_{n}}(y)\}.$ Fix $r\in
	\mathbb{R }$. From Lemma \ref{lem1_comp} there exists a subsequence $\{a_{n_{k}}\}$
	of $\{a_{n}\}$ and $z\in Y$ such that
	\begin{equation}
	f^{a_{n_{k}}}(z)\in F_{z_{0}}(\xi ,r)\mbox{ for }k\in \mathbb{N}\mbox{.}
	\label{bhl2}
	\end{equation}
	The remaining part of the proof is essentially the same as the proof of
	Theorem \ref{main}.
\end{proof}

Notice that Proposition \ref{A3star} in Section 4 is valid for a convex
domain of any Banach space and thus we have the following counterpart of
Lemma \ref{A3}.

\begin{lemma}
	\label{A3'}Suppose that $D$ is a strictly convex domain of a Banach space $V$
	and $(D,d)$ is a metric space satisfying condition (\ref{C}) whose topology
	coincides with the norm topology. Then $(D,d)$ satisfies Axiom 3'.
\end{lemma}

Theorem \ref{main_comp} and Lemma \ref{A3'} now yields the following
extension of Theorem \ref{main_Rn}.

\begin{theorem}
	\label{main_Rn_comp}Suppose that $D$ is a strictly convex domain of a Banach
	space and $(D,d)$ is a $(1,\kappa )$-quasi-geodesic space satisfying Axiom
	1' (with respect to the norm closure $(\overline{D},\left\Vert \cdot
	\right\Vert )$) and condition $(C).$ If $f:D\rightarrow D$ is a compact
	nonexpansive mapping without fixed points, then there exists $\xi \in
	\partial D$ such that the iterates $f^{n}$ of $f$ converge uniformly on
	bounded sets of $D$ to $\xi .$
\end{theorem}

\begin{proof}
	Suppose first that $\{f^{n}(y)\}$ is bounded for some $y\in D$.
	Let\thinspace
	\[
	r=\inf_{z\in D}\limsup_{n\rightarrow \infty }d(z,f^{n}(y))
	\]
	and notice that the asymptotic center
	\[
	A=\{x\in D:\limsup_{n\rightarrow \infty }d(x,f^{n}(y))=r\}
	\]
	of $\{f^{n}(y)\}$ is nonempty. Indeed, $A=\bigcap_{\varepsilon
		>0}A_{\varepsilon }$, where
	\[
	A_{\varepsilon }=\{x\in D:\limsup_{n\rightarrow \infty }d(x,f^{n}(y))\leq
	r+\varepsilon \}
	\]
	and since $f$ is nonexpansive, $f(A_{\varepsilon })\subset A_{\varepsilon }.$
	Hence $\overline{f(A_{\varepsilon })}\subset \overline{A_{\varepsilon }}
	=A_{\varepsilon }$ since $A_{\varepsilon }$ is bounded, $d$-closed and,
	therefore, also $\left\Vert \cdot \right\Vert $-closed from Axiom 1'. Thus
	\[
	\emptyset \neq \bigcap_{\varepsilon >0}\overline{f(A_{\varepsilon })}\subset
	\bigcap_{\varepsilon >0}A_{\varepsilon }=A
	\]
	since $f$ is compact. Moreover, $A$ is bounded closed and convex, $
	f(A)\subset A$ and $\overline{f(A)}$ is compact. It follows from the
	Schauder fixed-point theorem that $f$ has a fixed point which is a
	contradiction.
	
	Therefore, $f:D\rightarrow D$ has unbounded orbits and then the conclusion
	immediately follows from Theorem \ref{main_comp}.
\end{proof}

As discussed in Section 2, every bounded convex domain $D$ of a Banach space
can be equipped with the Hilbert metric $d_{H}$ under which it becomes a
complete geodesic space. Moreover, $(D,d_{H})$ satisfies Axiom 1' (see,
e.g., \cite[Theorem 4.13]{Nu}) and condition $(C)$. Therefore, the following
corollary, which is a special case of \cite[Theorem 4.17]{Nu}, is a direct
consequence of Theorem \ref{main_Rn_comp}.

\begin{corollary}
	Let $D$ be a bounded strictly convex domain of a Banach space. If $
	f:(D,d_{H})\rightarrow (D,d_{H})$ is a fixed-point-free compact nonexpansive
	map then there exists $\xi \in \partial D$ such that the iterates $f^{n}$ of
	$f$ converge uniformly on bounded sets of $D$ to $\xi .$
\end{corollary}

The case of the Kobayashi metric $k_{D}$ is more delicate. It is well known
that if $D$ is a bounded convex domain of a complex Banach space then $
(D,k_{D})$ is a complete metric space that satisfies Axiom 1' (see \cite{Ha,
	KRS}) but, in general, it is not a geodesic space. However, it directly
follows from the Lempert characterization of the Kobayashi distance
(presented in Section 2) that $(D,k_{D})$ is a $(1,\varepsilon)$-quasi-geodesic space for any $\varepsilon >0.$ Thus we have the following
consequence of Theorem \ref{main_Rn_comp}.

\begin{corollary}
	Let $D$ be a bounded strictly convex domain of a complex Banach space. If $
	f:(D,k_{D})\rightarrow (D,k_{D})$ is a fixed-point-free compact nonexpansive
	map then there exists $\xi \in \partial D$ such that the iterates $f^{n}$ of
	$f$ converge uniformly on bounded sets of $D$ to $\xi $.
\end{corollary}

Compact holomorphic fixed-point-free mapppings of the open unit ball of a
Hilbert space were studied in \cite{ChMe}. Then the above result was proved
in \cite[Theorem 4.4]{KKR1}, in the case of the open unit ball of a
uniformly convex space, in \cite[Theorem 4.2]{BKR1}, in the case of a
bounded strictly convex domain of a reflexive space, and finally in
\cite[Theorem 4.1]{BKR2}.

\section{The fixed point property and unbounded sets}

Let $(Y,d)$ be a metric space and fix $w\in Y.$ The \textit{Gromov product} of $
x,y\in Y$ at $w$ is defined to be
\[
(x,y)_{w}=\frac{1}{2}(d(x,w)+d(y,w)-d(x,y)).
\]
Let $\delta \geq 0.$ A metric space $Y$ is said to be \textit{$\delta $-hyperbolic}
if
\[
(x,y)_{w}\geq \min \{(x,z)_{w},(y,z)_{w}\}-\delta
\]
for every $x,y,z,w\in Y$. If $Y$ is $\delta $-hyperbolic for some $\delta
\geq 0$ then $Y$ is called \textit{Gromov hyperbolic}. The Gromov product is used to
define an abstract boundary $Y(\infty )$ called the \textit{ideal (}or \textit{Gromov)
	boundary}. We say that a sequence $\{x_{n}\}$ converges to infinity if $
\lim_{n,m\rightarrow \infty }(x_{n},x_{m})_{w}=\infty $ for some (and hence
any) $w\in Y$. Define $Y(\infty )$ as the set of the equivalence classes of
sequences converging to infinity, where two such sequences $\{x_{n}\}$ and $
\{y_{n}\}$ are equivalent if $\lim_{n,m\rightarrow \infty
}(x_{n},y_{m})_{w}=\infty .$ If $Y$ is a proper geodesic and Gromov
hyperbolic space, then there is a natural metrizable topology on
$\overline{Y }=Y\cup Y(\infty )$ that makes it a compactification of $Y,$ and $Y$ is
open in $\overline{Y}$ (see e.g., \cite[Prop. III.3.7]{BrHa}).

\begin{lemma}
	If $Y$ is Gromov hyperbolic then it satisfies Axiom $3$ with respect to the
	Gromov boundary.
\end{lemma}

\begin{proof}
	Suppose that $\{x_{n}\}$ and $\{y_{n}\}$ are sequences in $Y$, $
	x_{n}\rightarrow \xi \in Y(\infty )$ and
	\[
	d(x_{n},y_{n})-d(y_{n},w)\rightarrow -\infty
	\]
	for some $w\in Y.$ Then $(x_{n},x_{m})_{w}\rightarrow \infty $ and
	\[
	(x_{m},y_{m})_{w}=\frac{1}{2}(d(x_{m},w)+d(y_{m},w)-d(x_{m},y_{m}))
	\rightarrow \infty .
	\]
	Hence
	\[
	(x_{n},y_{m})_{w}\geq \min \{(x_{n},x_{m})_{w},(y_{m},x_{m})_{w}\}-\delta
	\rightarrow \infty
	\]
	which means that $y_{n}\rightarrow \xi .$
\end{proof}

Therefore, if we recall Theorem \ref{main2}, we see at once that the
Wolff--Denjoy theorem holds for any complete locally compact geodesic space
$Y$ which is Gromov hyperbolic. It was Karlsson who observed that properness
of $Y$ is in fact not needed (see \cite[Prop. 5.1]{Ka1}).

\begin{theorem}
	\label{K}Let $Y$ be a Gromov hyperbolic space and suppose that $f$ is a
	nonexpansive mapping such that $d(y,f^{n}(y))\rightarrow \infty $ for some $
	y\in Y.$ Then for every $x\in Y$ the orbit $\{f^{n}(x)\}$ converges to a
	point $\xi $ on the Gromov boundary $Y(\infty ).$
\end{theorem}

\begin{proof}
	Let $\{f^{\varphi (i)}(y)\}$ be a subsequence of $\{f^{n}(y)\}$ such that $
	d(y,f^{k}(y))\leq d(y,f^{\varphi (i)}(y))$ for $k\leq \varphi (i), \ i=1,2,\ldots$.
	Then
	\begin{eqnarray*}
		(f^{k}(y),f^{\varphi (i)}(y))_{y} &=&\frac{1}{2}(d(f^{k}(y),y)+d(f^{\varphi
			(i)}(y),y)-d(f^{k}(y),f^{\varphi (i)}(y))) \\
		&\geq &\frac{1}{2}(d(f^{k}(y),y)+d(f^{\varphi (i)}(y),y)-d(y,f^{\varphi
			(i)-k}(y))) \\
		&\geq &\frac{1}{2}d(f^{k}(y),y)
	\end{eqnarray*}
	for each $k\leq \varphi (i).$ In particular,
	\[
	(f^{\varphi (i)}(y),f^{\varphi (j)}(y))_{y}\geq \frac{1}{2}\min
	\{d(f^{\varphi (i)}(y),y),d(f^{\varphi (j)}(y),y)\}\rightarrow \infty \mbox{
		as }i,j\rightarrow \infty ,
	\]
	and hence there exists $\xi \in Y(\infty )$ such that $f^{\varphi
		(i)}(y)\rightarrow \xi.$
	It follows that
	\begin{eqnarray*}
		(f^{i}(x),f^{\varphi (i)}(y))_{y} &\geq &(f^{i}(y),f^{\varphi
			(n)}(y))_{y}-2d(x,y) \\
		&\geq &\frac{1}{2}d(f^{i}(y),y)-2d(x,y)\rightarrow \infty \mbox{ as }
		i\rightarrow \infty
	\end{eqnarray*}
	for any $x\in Y.$ Therefore, $f^{n}(x)\rightarrow \xi \in Y(\infty ).$
\end{proof} 

If $Y$ is a geodesic and $\delta$-hyperbolic then,
given three points
$x,y,z\in Y\,$and $y^{\prime }\in \lbrack x,y], \ z^{\prime }\in \lbrack x,z]$
such that $d(x,y^{\prime })=d(x,z^{\prime })\leq (y,z)_{x}$, then
$d(y^{\prime },z^{\prime })\leq 4\delta $ (see, e.g., \cite[Prop. 2.1.3]{BuSch}).
A subset $K$ of a geodesic space $Y$
is said to be \textit{geodesically bounded} if it does not contain any
geodesic ray, that is, there does not exist a map $\sigma :[0,\infty
)\rightarrow K$ such that $d(\sigma (t_{1}),\sigma (t_{2}))=\left\vert
t_{1}-t_{2}\right\vert $ for all $t_{1},t_{2}\in \lbrack 0,\infty ).$ A
subset $K$ of $Y$ is said to be \textit{geodesically convex} if a geodesic $
[x,y]\subset K$ for any $x,y\in K$. In 2015, Pi\c{a}tek \cite{Pi1} obtained
an interesting characterization of geodesic boundedness in spaces of
negative curvature CAT$(\kappa )$ (see \cite{BrHa} for a thorough treatment
of these spaces). We show how to apply Karlsson's theorem \ref{K} to
simplify the arguments a little bit.

Recall that a geodesic space $(Y,d)$ is called \textit{Busemann convex} if $
d(z_{\alpha },z_{\alpha }^{\prime })\leq (1-\alpha )d(x,x^{\prime })+\alpha
d(y,y^{\prime })$ for every $x,y,x^{\prime },y^{\prime }\in Y$ and $\alpha
\in \lbrack 0,1]$, where $z_{\alpha },z_{\alpha }^{\prime }$ are points on
geodesic segments $[x,y],[x^{\prime },y^{\prime }]$, respectively, such that
$d(x,z_{\alpha })=\alpha d(x,y)$ and $d(x^{\prime },z_{\alpha }^{\prime
})=\alpha d(x^{\prime },y^{\prime })$. The following lemma was proved in
\cite[Lemma 2.3]{Pi2}.

\begin{lemma}
	\label{P}Let $Y$ be a complete Busemann convex and $\delta $-hyperbolic
	space for some $\delta \geq 0.$ Then for each $w\in Y$ and a sequence $
	\{x_{n}\}$ of $Y$ such that $(x_{n},x_{m})_{w}\rightarrow \infty $ as $
	n,m\rightarrow \infty $, the geodesics $\sigma
	_{n}:[0,d(w,x_{n})]\rightarrow Y$ joining $w$ and $x_{n}$ converge to some
	geodesic ray $\sigma :[0,\infty )\rightarrow Y.$
\end{lemma}

\begin{proof}
	Fix $r>0$ and select on each geodesic segment $[w,x_{n}]$ such that $
	d(w,x_{n})>r$ a point $u_{n}$ such that $d(w,u_{n})=r.$ We show that $
	\{u_{n}\}$ is a Cauchy sequence. If $\delta =0$ then from $0$-hyperbolicity $
	\{u_{n}\}$ is constant so we can assume that $\delta >0.$ Let $\varepsilon
	>0.$ Then $(x_{n},x_{m})_{w}>\frac{4\delta r}{\varepsilon }$ for sufficiently
	large $n,m\geq n_{0}$ and since $(x_{n},x_{m})_{w}\leq \max
	\{d(x_{n},w),d(x_{m},w)\}$ there exist points $z_{n}\in \lbrack
	w,x_{n}], \ z_{m}\in \lbrack w,x_{m}]$ such that $
	d(z_{n},w)=d(z_{m},w)=(x_{n},x_{m})_{w}$. It follows from $\delta$-hyperbolicity 
	that $d(z_{n},z_{m})\leq 4\delta $ and since $Y$ is Busemann
	convex we have
	\[
	d(u_{n},u_{m})\leq \frac{d(w,u_{n})}{d(w,z_{n})}d(z_{n},z_{m})\leq \frac{r}{
		(x_{n},x_{m})_{w}}4\delta <\varepsilon
	\]
	for $n,m\geq n_{0}.$ It follows that $\{u_{n}\}$ is a Cauchy sequence and
	thus for every $r>0$ there exists a limit $\sigma (r)=\lim_{n\rightarrow
		\infty }\sigma _{n}(r).$ Clearly,
	\[
	d(\sigma (t_{1}),\sigma (t_{2}))=\lim_{n\rightarrow \infty }d(\sigma
	_{n}(t_{1}),\sigma _{n}(t_{2}))=\left\vert t_{1}-t_{2}\right\vert
	\]
	for $t_{1},t_{2}\in \lbrack 0,\infty )$ and hence $\sigma $ is a geodesic
	ray.
\end{proof}

Combining Theorem \ref{K} with Lemma \ref{P} gives a simpler proof of
\cite[Theorem 4.1]{Pi1} (see also \cite{Pi2, Pi3} for more general results).

\begin{theorem}
	Suppose that $K$ is a nonempty closed and geodesically convex subset of a
	complete CAT$(\kappa )$ space with $\kappa <0.$ Then $K$ has the fixed point
	property for nonexpansive mappings if and only if $K$ is geodesically
	bounded.
\end{theorem}

\begin{proof}
	If $K$ is not geodesically bounded then there exists a geodesic ray $\sigma
	:[0,\infty )\rightarrow K$ and hence the composition of the metric
	projection $\pi $ of $K$ onto $\sigma ([0,\infty ))$ with the shift operator
	is a fixed-point-free nonexpansive mapping (see, e.g.,
	\cite[ Prop. II.2.4]{BrHa}).
	
	To prove the converse, suppose on the contrary that there is a fixed point
	free nonexpansive map $f:K\rightarrow K.$ Fix $y\in K$ and notice that $
	d(f^{n}(y),y)$ is unbounded since otherwise, the asymptotic center
	\[
	A=\{x\in K:\limsup_{n\rightarrow \infty }d(x,f^{n}(y))=\inf_{z\in
		K}\limsup_{n\rightarrow \infty }d(z,f^{n}(y))\}
	\]
	of a sequence $\{f^{n}(y)\}$ consists of a single point that is fixed under $
	f$, see e.g., \cite[Proposition 7]{DKS}. Therefore, there is a subsequence $
	\{f^{n_{i}}(y)\}$ of $\{f^{n}(y)\}$ such that $d(f^{n_{i}}(y),y)\rightarrow
	\infty $ as $i\rightarrow \infty $ and arguing as in the proof of Theorem %
	\ref{K} (with $f^{n_{i}}$ instead of $f^{n}$) we conclude that $
	(f^{n_{i}}(y),f^{n_{j}}(y))_{y}\rightarrow \infty $ as $i,j\rightarrow
	\infty .$ Since for every $\kappa <0$, CAT$(\kappa )$ is Busemann convex and
	Gromov hyperbolic (see, e.g., \cite[Prop. III.1.2]{BrHa}), it follows from
	Lemma \ref{P} that there is a geodesic ray $\sigma :[0,\infty )\rightarrow K$, which is a contradiction.
\end{proof}

\end{document}